\documentclass[a4paper,12pt,openany]{amsart}
\usepackage[PetersLenny]{fncychap}
\usepackage[latin1]{inputenc}
 \usepackage[T1]{fontenc} \usepackage[english ]{babel,minitoc}
 \usepackage[nice]{nicefrac}
\usepackage{hyperref,graphicx,fancyhdr,fancybox,rotating,xcolor,appendix}
\usepackage{ulem,hhline,dsfont,pifont,multicol,lmodern}
\usepackage{amsmath,amsthm,amsfonts,amsbsy,amssymb,geometry,times,pst-node}
\usepackage{euscript}
\usepackage{textcomp} \usepackage{sidecap}
\usepackage{pdfpages} \usepackage{setspace}
\newcommand{\norm}[1]{\left\Vert#1\right\Vert}
\newcommand{\scal}[1]{\left<#1\right>}
\newtheorem{theorem}{Theorem}[section]

\newtheorem{corollary}[theorem]{Corollary}

\newtheorem{remark}[theorem]{Remark}
\newcommand{\V}{\mathcal{V}} 
\newcommand{\Vs}{\mathcal{V}}
\newcommand{\N}{\mathbb{N}}

\newcommand{\R}{\mathbb{R}}      \newcommand{\Rd}{\mathbb{R}^d}

\newcommand{\C}{\mathbb{C}}      \newcommand{\Cd}{\mathbb{C}^{d}}
\newcommand{\F}{\mathcal{F}}
      	  \newcommand{\lrd}{L^{1}(\Rd)}    
      \newcommand{\ldrd}{L^{2}(\Rd)}   \newcommand{\ldrdd}{L^{2}(\Cd)}
	   \newcommand{\Srd}{\mathcal{S}(\Rd)} \newcommand{\Scd}{\mathcal{S}(\Cd)}

\newcommand{\bz}{\overline{z}}

\begin{document}

\title{Complex Hermite functions as Fourier-Wigner transform}
\author{F. Agorram}     
\author{A. Benkhadra}   
\author{A. El Hamyani} 
\author{A. Ghanmi}     
 \address{P.D.E. and Spectral Geometry,
          Laboratory of Analysis and Applications - URAC/03,
          Department of Mathematics, P.O. Box 1014,  Faculty of Sciences,
          Mohammed V University of Rabat, Morocco}
\maketitle

\begin{abstract}
We prove that the complex Hermite polynomials $H_{m,n}$ on the complex plane $\C$ can be realized as the Fourier-Wigner transform $\V$ of the well-known real Hermite functions $h_n$ on real line $\R$. This reduces considerably the Wong's proof \cite[Chapter 21]{Wong} giving the explicit expression of $\V(h_m,h_n)$ in terms of the Laguerre polynomials. Moreover, we derive a new generating function for the $H_{m,n}$ as well as some new integral identities.
\bigskip


\end{abstract}

\section{Introduction} \label{s1}

The so-called Fourier-Wigner transform is the windowed Fourier transform defined by
\begin{align}\label{FWT}
 \V(f,g)(p,q)=\left(\frac{1}{2\pi}\right)^{\frac{d}{2}}   \int_{\Rd}  e^{i \scal{x+\frac{p}{2},q}} f(x+p) \overline{g(x)} dx
\end{align}
for every $(p,q)\in \Rd \times \Rd$ and every complex-valued functions $f,g \in \ldrd$.
This transform is a basic tool to study the Weyl transform \cite{Folland,Thangavelu,Wong}
and to interpret quantum mechanics as a form of nondeterministic statical dynamics \cite{Moyal}.
It is also used to study the nonexisting joint probability distribution of positioned momentum in a given state \cite{Wong}.
In addition, the transform $\V$ leads to the well-known Segal-Bargmann transform for a special window function  \cite{Folland,Thangavelu93}.
As basic property of this transform, one can use it to construct orthonormal bases for the Hilbert space $L^2(\Cd)$ from orthonormal bases
of $L^2(\Rd)$. 

In the present paper, we provide a new application of the Fourier-Wigner transform in the context of the complex Hermite polynomials $H_{m,n}$
\cite{Gh13ITSF,Ismail13b,DunklXu14}.
More precisely, we realize $H_{m,n}(z;\bar z)$ as the Fourier-Wigner transform of the well-known real Hermite
functions $h_n$ on $\R$.
This reduces considerably the Wong's proof \cite[Chapter 21]{Wong} giving the explicit expression of $\V(h_n,h_m)$ in terms of the Laguerre polynomials.
Moreover, we derive a new generating function for the complex Hermite polynomials $H_{m,n}$ as well as some new identities in the context of integral calculus.

The basic topics that we need in Fourier-Wigner transform, and in real and complex Hermite polynomials are collected in Section 2 and Section 3, respectively.
In Section 4, we state and prove our main results. We end the paper with some concluding remarks.

\section{The Fourier-Wigner transform} \label{s2}

The Fourier-Wigner transform  $\V: (f,g) \longmapsto  \V(f,g)$, given through \eqref{FWT}, is a well defined bilinear mapping
on $\ldrd \times \ldrd$, with
$$|\V(f,g)(p,q)|  \leq \left(\dfrac{1}{2\pi}\right)^{\frac{d}{2}} \norm{f}_{\ldrd} \norm{g}_{\ldrd}$$
for every $(f,g)\in \ldrd \times \ldrd$ and $(p,q)\in\Rd\times\Rd$.
It can be realized as a Fourier transform defined on $\Rd$ by
$$
\F(f)(\xi):=\left(\frac{1}{2\pi}\right)^{\frac{d}{2}} \int_{\Rd} e^{- i \scal{y,\xi}}f(y)dy.
$$
In fact, we have $ \V(f,g)(p,q) =   \F(K_{f,g}(\cdot|p))(- q),$
where the function $y \longmapsto K_{f,g}(y|p)$ belonging to $\lrd$ is defined on $\Rd$  by
\begin{align}\label{fct:Kfg}
 K_{f,g}(y|p)=f\left(y+\frac{p}{2}\right)\overline{g\left(y-\frac{p}{2}\right)}
\end{align}
for every given $f,g \in \ldrd$ and fixed $p\in\Rd$.
 More explicitly,
\begin{align}\label{Exp:Vs}
\V(f,g)(p,q)=\left(\frac{1}{2\pi}\right)^{\frac{d}{2}}
\int_{\Rd} e^{i \scal{y,q}}f\left(y+\frac{p}{2}\right)\overline{g\left(y-\frac{p}{2}\right)}dy.
\end{align}
An interesting result for $\V$ is the Moyal's formula
$$\scal{\V(f,g),\V(\varphi,\psi)}_{L^{2}(\Cd)}= \scal{f,\varphi}_{\ldrd} \scal{\psi,g}_{\ldrd}; \quad f,g,\varphi,\psi \in \ldrd,$$
giving rise to
$$\norm{\V(f,g)}^{2}_{L^{2}(\Cd)}= \norm{f}^{2}_{\ldrd}\norm{g}^{2}_{\ldrd}.$$
 Subsequently, we have $\V(\ldrd \times \ldrd) \subset L^{2}(\Cd).$
Moreover, $ \V(\Srd\times \Srd) \subset \Scd$ for every $f$ and $g$ in the Schwartz space $\Srd$.

Using Moyal's formula and, it can be shown \cite{Wong} that the Fourier-Wigner $\V$ can be used to construct orthogonal bases of $\ldrdd$ from those of $\ldrd$.
More precisely, if $\{\varphi_{k},k\in\N\}$ is an orthonormal basis of $\ldrd$,
then  $\{\varphi_{jk}=\V(\varphi_{j},\varphi_{k}); \, j,k\in\N\}$ is an orthonormal basis of $\ldrdd$ with
$$
\norm{ \varphi_{jk} }_{\ldrdd}  =  \norm{\varphi_{j}}_{\ldrd}^2\norm{\varphi_{k}}_{\ldrd}^2.
$$

\section{Real and complex Hermite polynomials} \label{s3}

The classical real Hermite polynomials $H_n(x)$ are defined by the Rodrigues' formula
$$
H_{n}(x)=(-1)^{n}e^{x^{2}}\dfrac{d^{n}}{dx^{n}}(e^{-x^{2}}),
$$
or also by their generating function
\begin{align}\label{GenFhn}
\sum_{m=0}^{+\infty}\dfrac{H_{m}(x)}{m!}t^{m}=e^{-t^{2}+2xt}.
\end{align}
Associated to $H_n$ they are the Hermite functions
$$
h_{n}(x)=e^{-\frac{x^{2}}{2}}H_{n}(x),
$$
which constitute an orthogonal basis of the Hilbert space $L^{2}(\R)$.
An interesting result for the Hermite functions is the Mehler's formula \cite{Szego75,Ismail2014},
\begin{align}\label{Mehler}
\sum_{m=0}^{+\infty}  \frac{\lambda^m}{2^{m} m!} h_{m}(x) h_{m}(y) = g(x,y|\lambda)
\end{align}
 fulfilled for $|\lambda|<1$, where
\begin{align}\label{Mehlerfctg}
 g(x,y|\lambda) = \frac{1}{\sqrt{1-\lambda^2}}
 \exp\left( -\frac{1+\lambda^2}{2(1-\lambda^2)}  (x^2+y^2) +   \frac{2\lambda}{1-\lambda^2} xy \right).
\end{align}

 An extension of $H_{m}(x)$ to the complex variable are the so-called complex Hermite polynomials $H_{m,n}(z,\bz )$
 introduced by It\^o \cite{Ito51} in the context of complex Markov process.
 They can be defined by means of their generating function \cite{Gh13ITSF}
\begin{align}\label{GenHmn}
 \sum_{m,n=0}^{+\infty}\frac{u^{m}}{m!}\frac{v^{n}}{n!}H_{m,n}(z,\overline{z})
 = e^{ - uv + z u + \overline{z} v}.
 \end{align}
 The explicit expression of $H_{m,n}$ in terms of the generalized Laguerre polynomials (\cite{Szego75})
 $L^{(\alpha)}_n(x)$ is given by \cite{IntInt06,Gh13ITSF}
       \begin{align}\label{HmnLaguerre}
             H_{m,n}(z,\bz ) =  (-1)^{\min(m,n)} (\min(m,n))! |z|^{|m-n|} e^{i(m-n)\arg(z)}  L^{(|m-n|)}_{\min(m,n)}(|z|^2)
       \end{align}
with $z=|z|e^{i\arg(z)}$.
The polynomials $H_{m,n}$ constitute a complete orthogonal system of the Hilbert space $L^2(\C;e^{-|z|^2}d\lambda)$ and
 appear naturally when investigating the eigenvalue problem of some second order differential
 operators of Laplacian type \cite{Shigekawa87,Matsumoto96,GhInJMP}.

Several interesting features of $H_{m,n}$ in connection with singular values
 of Cauchy transform \cite{IntInt06}, coherent states theory
 \cite{AliBagarelloHonnouvo10,AliBagarelloGazeau13},
 combinatory \cite{Ismail13a,Ismail13b} and signal processing
 \cite{RaichZhou04,DallingerRuotsalainenWichmanRupp10} have been studied recently.
 In the next section, we realize $H_{m,n}(z;\bar z)$ as the Fourier-Wigner transform of the well-known real Hermite
functions $h_n$ on $\R$, and derive interesting identities of these polynomials.

\section{Main results} \label{s4}

In this section, we provide new applications of the Fourier-Wigner transform and see how it turns up in the context of complex Hermite polynomials.
The first main result gives the explicit expression of $\Vs(h_{m},h_{n})$ in terms of the
complex Hermite polynomials $H_{m,n}$. Namely, we assert

\begin{theorem}\label{thm:hVH}
For every $p,q\in\R$, we have
\begin{align}\label{hVH}
\Vs(h_{m},h_{n})(p,q) = (-1)^{n}\sqrt{2}^{m+n-1}e^{-\frac{p^{2}+q^{2}}{4}} H_{m,n}\left(\frac{p+iq}{\sqrt{2}},\frac{p-iq}{\sqrt{2}}\right).
\end{align}
\end{theorem}

\begin{proof}
Note first that by making use of the generating function \eqref{GenFhn} for the real Hermite polynomials $H_{m}$, we get easily that
\begin{align}\label{GenFhnhm} \sum_{m,n=0}^{+\infty}
\frac{u^{m}}{m!} \frac{v^{n}}{n!} H_{m}\left(y+\frac{p}{2}\right) H_{n}\left(y-\frac{p}{2}\right)
= e^{-(u^2+v^2)+2y(u+v) +p(u-v)} .
\end{align}
On the other hand, according to the expression of $\Vs$ given by \eqref{Exp:Vs}, we can write
\begin{align*}
 \Vs(h_{m},h_{n})(p,q)&=\dfrac{1}{\sqrt{2\pi}}\int_{-\infty}^{+\infty}e^{iyq} e^{-\frac{(y+\frac{p}{2})^{2}}{2}}e^{-\frac{(y-\frac{p}{2})^{2}}{2}} H_{m}\left(y+\frac{p}{2}\right)H_{n}\left(y-\frac{p}{2}\right)dy\\
 &=\dfrac{e^{-\frac{p^{2}}{4}}}{\sqrt{2\pi}}  \int_{-\infty}^{+\infty}e^{-y^{2}+iyq} H_{m}\left(y+\frac{p}{2}\right)H_{n}\left(y-\frac{p}{2}\right) dy.
\end{align*}
By means of \eqref{GenFhnhm}, we obtain
\begin{align*}
\sum_{m,n=0}^{+\infty}\frac{u^{m}}{m!}\frac{v^{n}}{n!} \Vs(h_{m},h_{n})(p,q)
=\dfrac{e^{-\frac{p^{2}}{4}}}{\sqrt{2\pi}}  e^{-(u^2+v^2) + p(u-v)}
\int_{-\infty}^{+\infty} e^{-y^2 + (iq+2(u+v))y } dy.
\end{align*}
More explicitly, we have
\begin{align}\label{Gen}
\sum_{m,n=0}^{+\infty}\frac{u^{m}}{m!}\frac{v^{n}}{n!} \Vs(h_{m},h_{n})(p,q)
=\dfrac{e^{-\frac{p^{2}+q^{2}}{4}}}{\sqrt{2}}  e^{2uv +(p+iq)u - (p-iq)v}
\end{align}
which follows by applying the well-known expression of the classical Gauss integral given by
\begin{align}\label{Gauss}
 \int_{\R} e^{-\alpha y^2 + \beta y } dy = \left(\frac{\pi}{\alpha}\right)^{1/2} e^{\beta^2/4\alpha}
\end{align}
with $\alpha>0$ and $\beta\in\C$. Now, by setting  $z=p+iq$, we get
\begin{align}\label{Gen}
\sum_{m,n=0}^{+\infty}\frac{u^{m}}{m!}\frac{v^{n}}{n!} \Vs(h_{m},h_{n})(p,q)
=\dfrac{e^{-\frac{|z|^{2}}{4}}}{\sqrt{2}}  e^{2uv + z u - \overline{z} v}.
\end{align}
In the right-hand-side of the last equality, we recognize the generating function \eqref{GenHmn} of the complex Hermite polynomials
$H_{m,n}$. Therefore,
\begin{align*}
\sum_{m,n=0}^{+\infty}\frac{u^{m}}{m!}\frac{v^{n}}{n!} \Vs(h_{m},h_{n})(p,q)
&=\dfrac{e^{-\frac{|z|^{2}}{4}}}{\sqrt{2}}
\sum_{m,n=0}^{+\infty}
\frac{(\sqrt{2}u)^{m}}{m!}\frac{(-\sqrt{2}v)^{n}}{n!}
H_{m,n}\left(\frac{z}{\sqrt{2}},\frac{\overline{z}}{\sqrt{2}}\right)\\
&=\dfrac{e^{-\frac{|z|^{2}}{4}}}{\sqrt{2}}
\sum_{m,n=0}^{+\infty} \frac{u^{m}}{m!}\frac{v^{n}}{n!}
 \left[(-1)^{n} \sqrt{2}^{m+n}H_{m,n}\left(\frac{z}{\sqrt{2}},\frac{\overline{z}}{\sqrt{2}}\right)\right].
\end{align*}
By identifying the two power series, we obtain \eqref{hVH}.
\end{proof}

The special case of $p=0$ in \eqref{hVH} yields the following

\begin{corollary} For every $t\in\R$, we have
\begin{align}\label{fhnhm}
\int_{-\infty}^{+\infty} H_{m}(y) H_{n}(y) e^{-y^2 - i t y} dy
 = (-1)^{n} \sqrt{\pi}\sqrt{2}^{m+n} e^{-\frac{t^{2}}{4}} H_{m,n}\left(-\frac{it}{\sqrt{2}},\frac{it}{\sqrt{2}}\right).
\end{align}
\end{corollary}

While when specifying $q=0$, we can deduce the following

\begin{corollary} For every $t\in\R$, we have
\begin{align}\label{inthnhm}
\int_{-\infty}^{+\infty} H_{m}\left(y+\frac{t}{2}\right) H_{n}\left(y-\frac{t}{2}\right) e^{-y^2} dy = (-1)^n\sqrt{\pi}\sqrt{2}^{m+n}
H_{m,n}\left(\frac{t}{\sqrt{2}},\frac{t}{\sqrt{2}}\right).
\end{align}
\end{corollary}

\begin{remark}
By taking $t=0$ in \eqref{fhnhm} and \eqref{inthnhm}, we recover the well-known formula
\begin{align*}
\int_{-\infty}^{+\infty} H_{m}\left(y\right) H_{n}\left(y\right) e^{-y^2} dy = \sqrt{\pi} 2^{m} m! \delta_{m,n}
\end{align*}
since $H_{m,n}\left(0,0\right)= (-1)^m m! \delta_{m,n}$.
\end{remark}

As an immediate consequence of Theorem \ref{thm:hVH}, combined with the fact that $\Vs$
transforms orthogonal basis of $L^{2}(\R)$
to an orthogonal basis of $L^{2}(\C)$, we recover the following well-known result (see \cite{IntInt06,Gh08,Gh13ITSF,Ismail13a,Ismail13b}).

\begin{corollary}
The complex Hermite polynomials $H_{m,n}$ constitute an orthogonal
basis of the Hilbert space $L^{2}(\C,e^{-|z|^2}dxdy)$; $z=x+iy$.
\end{corollary}

In the sequel, we will investigate further consequences of Theorem \ref{thm:hVH}. We begin, by establishing a new generating
function for the complex Hermite polynomials $H_{m,n}$ with $m=n$. Namely, we assert

\begin{theorem}
For every  positive real number $0<\lambda<1$, we have
\begin{align}\label{GenMehleH}
\sum_{m=0}^{+\infty}  \frac{  \lambda ^{m}}{m!} H_{m,m}\left( z , \overline{z} \right)
  = \frac{ e^{\frac{\lambda }{1+\lambda} |z|^2 } }{1+\lambda} .
  \end{align}
\end{theorem}

\begin{proof}
Starting from \eqref{hVH}, we can write
\begin{align}\label{hVH1}
 H_{m,m}\left( z, \bz \right) 
 =
  (-1)^{m}\frac{\sqrt{2}}{2^{m}} e^{\frac{|z|^{2}}{2}} \Vs(h_{m},h_{m})(\sqrt{2}p,\sqrt{2}q)
\end{align}
with $z=p+iq$; $p,q\in\R$. Therefore, we get
\begin{align*}
\sum_{m=0}^{+\infty}  \frac{ \lambda^{m}}{m!} H_{m,m}\left( z,\bz \right)
 &=  \sqrt{2} e^{\frac{|z|^{2}}{2}} \sum_{m=0}^{+\infty}   \frac{(-\lambda)^{m}}{ 2^{m} m!}   \Vs(h_{m},h_{m})(\sqrt{2}p,\sqrt{2}q) \\
  &\stackrel{\eqref{Exp:Vs}}{=} \frac{e^{\frac{|z|^{2}}{2}}}{\sqrt{\pi}}
   \sum_{m=0}^{+\infty}   \frac{(-\lambda)^{m}}{ 2^{m} m!}
   \int_\R e^{i\sqrt{2}q y} h_{m}\left(y+\frac{p}{\sqrt{2}}\right) h_{m}\left(y-\frac{p}{\sqrt{2}}\right)  dy\\
   &= \frac{e^{\frac{|z|^{2}}{2}}}{\sqrt{\pi}}
    \int_\R e^{i\sqrt{2}q y} \left( \sum_{m=0}^{+\infty}   \frac{(-\lambda)^{m}}{ 2^{m} m!}
    h_{m}\left(y+\frac{p}{\sqrt{2}}\right) h_{m}\left(y-\frac{p}{\sqrt{2}}\right) \right) dy.
\end{align*}
In the last equality we recognize the Mehler's formula \eqref{Mehler} for the Hermite functions. Whence, it follows
\begin{align}\label{hVH2}
\sum_{m=0}^{+\infty}  \frac{ \lambda^{m}}{m!} H_{m,m}\left( z , \bz \right)
= \frac{e^{\frac{|z|^{2}}{2}}}{\sqrt{\pi}}  \int_\R  e^{i\sqrt{2}q y}
      g\left(y+\frac{p}{\sqrt{2}},y-\frac{p}{\sqrt{2}} \big| -\lambda\right)  dy.
\end{align}
Now, since
\begin{align*}
   g\left(y+\frac{p}{\sqrt{2}},y-\frac{p}{\sqrt{2}} \big| -\lambda\right)
    = \frac{1}{\sqrt{1-\lambda^2}} e^{ - \frac{1+ \lambda }{1-\lambda} y^2 - \frac{1-\lambda }{2(1+\lambda)} p^2 }.
\end{align*}
Equation \eqref{hVH2} becomes
\begin{align*}
\sum_{m=0}^{+\infty}  \frac{ \lambda^{m}}{m!} H_{m,m}\left( z , \bz \right)
 = \frac{e^{\frac{|z|^{2}}{2} - \frac{1-\lambda }{1+\lambda} \frac{p^2}{2}}}{\sqrt{\pi}\sqrt{1-\lambda^2}}
 \int_\R e^{- \frac{1+\lambda}{1-\lambda}y^2 + i\sqrt{2}q y} dy.
\end{align*}
Making appeal to \eqref{Gauss} with  $\alpha=\frac{1+\lambda}{1-\lambda}>0$, for $0<\lambda<1$, and $\beta=i\sqrt{2}q$, we get \eqref{GenMehleH}.
\end{proof}

\section{Concluding remarks} \label{s3}

Instead of $h_n$, we consider the orthonormal basis of $L^2(\C)$ given by
$$e_n(x) = \frac{h_n(x)}{\norm{h_n}} =  \frac{h_n(x)}{\sqrt{2^n n! \sqrt{\pi}}}.$$
Thus according to \eqref{HmnLaguerre}, giving the expression of the complex Hermite polynomials in terms of the
 generalized Laguerre polynomials $L^{(\alpha)}_n(x)$, we can rewrite the result of Theorem \ref{thm:hVH} as
 \begin{align*}
 \Vs(e_{m},e_{n})(p,q)
   =  \frac{1}{\sqrt{2\pi}} e^{-\frac{|z|^{2}}{4}}
  \left\{\begin{array}{lll}
     \dfrac{\sqrt{n!}}{\sqrt{m!}} \dfrac{z^{m-n}}{\sqrt{2}^{m-n} } L^{(m-n)}_{n}\left(\dfrac{|z|^2}{2}\right)   & \quad \mbox{if } m\geq n\\
     \quad \\
     (-1)^{n+m}\dfrac{\sqrt{m!}}{\sqrt{n!}} \dfrac{{\overline{z}}^{n-m}}{\sqrt{2}^{n-m} }   L^{(n-m)}_{m}\left(\dfrac{|z|^2}{2}\right) & \quad \mbox{if }  n\geq m
   \end{array} \right. ,
   \end{align*}
where $p+iq=z$. This reads equivalently as
  \begin{align}
   & \Vs(e_{j+k},e_{j})(p,q) =    \left(\frac{j!}{2\pi {2}^{k} (j+k)!}\right)^{1/2}
                  z^{k}   L^{(k)}_{j}\left(\frac{|z|^2}{2}\right) e^{-\frac{|z|^{2}}{4}}    \label{hVL1a}\\
   &\Vs(e_{j},e_{j+k})(p,q) =   (-1)^{k} \left(\frac{j!}{2\pi {2}^{k} (j+k)!}\right)^{1/2}
     {\overline{z}}^{k} L^{(k)}_{j}\left(\frac{|z|^2}{2}\right) e^{-\frac{|z|^{2}}{4}} .  \label{hVL2b}
  \end{align}
Whence, we recover the result established by Wong in \cite{Wong} giving the explicit expression of $\V(e_n,e_m)$ in terms of the Laguerre polynomials ($i\bz$ in Wong's notation is $z$ in ours). However, our proof reduces considerably the one given by Wong.

We conclude this paper by noting that by adopting the same approach as above, one can introduce a new class of orthogonal polynomials on the quaternion $\R^4=\mathbb{H}$. This is the subject of another paper.

\end{document}